\documentclass[10pt, a4paper]{article}
\usepackage{amssymb,amsmath,amsthm,hyperref,enumerate}
\usepackage[mathscr]{eucal}

\DeclareMathOperator{\Aut}{Aut}  
\DeclareMathOperator{\cl}{cl}   

\newcommand{\abar}{{\ensuremath{\bar{a}}}}
\newcommand{\bbar}{{\ensuremath{\bar{b}}}}
\newcommand{\cbar}{{\ensuremath{\bar{c}}}}

\newcommand{\ybar}{{\ensuremath{\bar{y}}}}

\DeclareMathOperator{\tp}{tp}  

\newcommand{\N}{\ensuremath{\mathbb{N}}}

\newcommand{\Q}{\ensuremath{\mathbb{Q}}}

\newcommand{\Loo}{\ensuremath{L_{\omega_1,\omega}}} 
\newcommand{\Looq}{\ensuremath{\Loo(Q)}}

\renewcommand{\phi}{\varphi}
\renewcommand{\le}{\ensuremath{\leqslant}}
\renewcommand{\ge}{\ensuremath{\geqslant}}

\newcommand{\class}[2]{\ensuremath{\left\{ #1 \,\left|\, #2 \right.\right\}}}

\newcommand{\into}{\hookrightarrow}

\newcommand{\subs}{\subseteq} 

\newcommand{\minus}{\ensuremath{\smallsetminus}}

\newcommand{\nstrong}{\ensuremath{\not\kern-4pt\lhd\;}} 

\newcommand{\cross}{\ensuremath{\times}}

\newbox\noforkbox \newdimen\forklinewidth
\forklinewidth=0.3pt \setbox0\hbox{$\textstyle\smile$}
\setbox1\hbox to \wd0{\hfil\vrule width \forklinewidth depth-2pt
  height 10pt \hfil}
\wd1=0 cm \setbox\noforkbox\hbox{\lower 2pt\box1\lower
2pt\box0\relax}
\def\unionstick{\mathop{\copy\noforkbox}\limits}

\def\nonfork_#1{\unionstick_{\textstyle #1}}

\setbox0\hbox{$\textstyle\smile$} \setbox1\hbox to \wd0{\hfil{\sl
/\/}\hfil} \setbox2\hbox to \wd0{\hfil\vrule height 10pt depth
-2pt width
                \forklinewidth\hfil}
\newbox\doesforkbox
\setbox\doesforkbox\hbox{\lower 2pt\box1 \lower
2pt\box2\lower2pt\box0\relax}
\def\nunionstick{\mathop{\copy\doesforkbox}\limits}

\def\fork_#1{\nunionstick_{\textstyle #1}}

\newtheorem{prop}{Proposition}[section]
\newtheorem{cor}[prop]{Corollary}
\newtheorem{theorem}[prop]{Theorem}
\newtheorem{lemma}[prop]{Lemma}

\newtheorem{claim}{Claim}

\theoremstyle{definition}
\newtheorem{defn}[prop]{Definition}

\newtheorem{remark}[prop]{Remark}

\newcommand{\closed}{\ensuremath{\preccurlyeq_{\mathrm{cl}}}}
\newcommand{\bdy}{\ensuremath{\partial}}

\newcommand{\qps}{quasiminimal pregeometry structure}

\newcommand{\monster}{\mathfrak{M}}

\newcommand{\K}{\mathcal{K}}

\newcommand{\aaa}{\bar{a}}
\newcommand{\bee}{\bar{b}}
\newcommand{\cee}{\bar{c}}
\newcommand{\dee}{\bar{d}}

\newcommand{\notindep}{\mathrel{\lower0pt\hbox to 3pt{\kern3pt$\not$\hss}\downarrow}}

\newcommand{\absnotion}{\downarrow^{abs}}
\newcommand{\notabsnotion}{\mathrel{\lower0pt\hbox to 3pt{\kern3pt$\not$\hss}\absnotion}}

\newcommand{\concat}{{}^\frown}

\title{Quasiminimal structures and excellence}
\author{Martin Bays, Bradd Hart, Tapani Hyttinen,\\Meeri Kes\"al\"a and Jonathan Kirby}
\date{\today}

\begin{document}

\maketitle

\begin{abstract}
  We show that the excellence axiom in the definition of Zilber's quasiminimal
  excellent classes is redundant, in that it follows from the other axioms.
  This substantially simplifies a number of categoricity proofs.
\end{abstract}

\section{Introduction}

{
\renewcommand{\thefootnote}{}
\footnotetext{The fourth author was funded by the Academy of Finland, project
number 1251557.\\
Published in the Bulletin of the London Mathematical Society 2013
doi:10.1112/blms/bdt076.
}

The notion of a quasiminimal excellent class was introduced by Boris Zilber in
\cite{Zilber05qmec} in order to prove categoricity of his non-elementary
theories of covers of the multiplicative group of a field (group covers)
\cite{Zilber06covers} and of pseudoexponential fields \cite{Zilber05peACF0}.
The excellence axiom is the most technical part, and is adapted from Shelah's
work on excellent sentences of $\Loo$ \cite{Sh87a}. Both Shelah's and Zilber's
work on excellence is described in Baldwin's monograph
\cite{Baldwin_Categoricity}. Zilber's original proof
of categoricity of group covers contained a gap, which was corrected in \cite{BZ11} by
strengthening a hypothesis in one of the statements relating to excellence and
giving a new proof. However, the proof of the categoricity of
pseudoexponential fields relied on the original stronger and now unproved
statement from \cite{Zilber06covers}. A patch for the categoricity proof for
pseudoexponential fields was recently circulated by the first and fifth
authors \cite{BK12patch}.

In this paper we show that the excellence axiom of quasiminimal excellence
classes is actually redundant, in that it follows from the other axioms. This
substantially simplifies the proof of categoricity of Zilber's group covers
and pseudoexponential fields, and avoids the troublesome part of the proofs
where the gaps were.

In the case of first-order theories, part of Shelah's Main Gap theorem
involves reducing a condition on $n$-systems of models, akin to excellence, to
the case $n=2$, where it becomes the condition (PMOP) that primary models
exist over independent pairs of models
(\cite{ShelahClassificationTheory},\cite{HartNOTOP}). The main insight behind
the current paper is that these arguments, suitably modified, apply also to
the (non-elementary) classes of structures considered here - and moreover that
the reduction can be pushed even further, to $n=1$, where the condition
becomes one of $\aleph_0$-stability over models. This reduction is performed in
Proposition~\ref{excellence}. In Propositions~\ref{splittaus} and
\ref{isolation in closure}, we find that this $\aleph_0$-stability condition
does follow from the $\aleph_0$-homogeneity over models assumed of
quasiminimal excellent classes. This argument is based on a classical argument
from stability theory, but the version in this paper is a modification of a
corresponding argument in the non-elementary framework of finitary AECs
\cite{indepInLocalAEC}.

An uncountable structure $M$ is {\em quasiminimal} if every first-order
$M$-definable subset of $M$ is countable or co-countable. In
section~\ref{sec:qm}, we consider in the light of our main results the question
of when a quasiminimal structure belongs to a quasiminimal excellent class.

Our main results directly answer Question~1 in \cite[Section~6]{OQMEC}.
They also render Question~2 there redundant: it asks for equivalence of the
excellence axiom and the conclusion of \cite[Lemma~3.2]{OQMEC}, which we show
both to be consequences of the other axioms, hence trivially equivalent
modulo them. The remaining questions, 3-5, concern finite-dimensional models;
our techniques say little about these, and in fact it is key to the proof of
our main result that we deal only with infinite-dimensional models.

The authors would like to thank the Max Planck Institute for Mathematics,
Bonn, where some of this work was done.

\section{Statement of main result}

Throughout this paper, the notion of type will be quantifier-free $L$-type,
denoted by $\tp$. It will follow from our axioms that if finite tuples $\abar$ and
$\bbar$ from a model satisfy the same quantifier-free $L$-type then they
satisfy the same complete type (and even the same $L_{\infty,\omega}$-type),
justifying our notation. In applications this is usually achieved by expanding the
language. However it does not necessarily follow that the
first-order theory of our models has quantifier-elimination, since not all
types of the first-order theory are necessarily realised in the models we
consider.

\begin{defn}
  Let $M$ be an $L$-structure for a countable language $L$, equipped with a
  pregeometry $\cl$ (or $\cl_M$ if it is necessary to specify $M$). We say that
  $M$ is a \emph{quasiminimal pregeometry structure} if the following hold:
  \begin{enumerate}[QM1.]
    \item The pregeometry is determined by the language. That is, if
      $\tp(a,\bbar) = \tp(a',\bbar')$ then $a \in \cl(\bbar)$ if and only if $a'
      \in \cl(\bbar')$.
    \item $M$ is infinite-dimensional with respect to $\cl$.
    \item (Countable closure property) If $A \subs M$ is finite then $\cl(A)$
      is countable.
    \item (Uniqueness of the generic type) Suppose that $H, H' \subs M$ are
      countable closed subsets, enumerated such that $\tp(H)=\tp(H')$. If $a
      \in M \minus H$ and $a' \in M \minus H'$ then  $\tp(H,a) = \tp(H',a')$
      (with respect to the same enumerations for $H$ and $H'$).
    \item ($\aleph_0$-homogeneity over closed sets and the empty set)
      \ \\ Let $H, H' \subs M$ be countable closed subsets or empty,
      enumerated such that $\tp(H)=\tp(H')$, and let $\bbar,\bbar'$ be finite
      tuples from $M$ such that $\tp(H,\bbar) = \tp(H',\bbar')$, and let $a
      \in \cl(H,\bbar)$. Then there is $a'\in M$ such that $\tp(H,\bbar,a) =
      \tp(H',\bbar', a')$.
  \end{enumerate}
  We say $M$ is a \emph{weakly quasiminimal pregeometry structure} if it
  satisfies all the axioms except possibly QM2.
\end{defn}

Note that, while in QM5 there is a restriction that $a \in \cl(H \bbar)$, in the presence of the other axioms this restriction can be removed. In particular we have the usual notion of $\aleph_0$-homogeneity of a structure:
\begin{lemma}\label{homogeneity over emptyset}
Let $M$ be a weakly quasiminimal pregeometry structure, let $\bbar,\bbar'$ be finite tuples from $M$ such that $\tp(\bbar) = \tp(\bbar')$, and let $\abar$ be a finite tuple from $M$. Then there is $\abar'$ in $M$ such that $\tp(\abar\bbar) = \tp(\abar'\bbar')$.
\end{lemma}
\begin{proof}
We may assume that $\abar$ is a singleton, $a$. If $a \in \cl(\bbar)$ then apply QM5. If not, since $\cl$ is a pregeometry and using QM1 we have $\dim M \ge \dim(a \bbar) = \dim(\bbar)+1 = \dim(\bbar')+1$. So there is $c \in M \minus \cl(\bbar')$, and by QM4 we can take $a'$ to be any such $c$.
\end{proof}

Given $M_1$ and $M_2$ both weakly quasiminimal pregeometry $L$-structures, we
say that an $L$-embedding $\theta: M_1 \into M_2$ is a \emph{closed embedding}
if for each $A \subs M_1$ we have $\theta(\cl_{M_1}(A))=\cl_{M_2}(\theta(A))$. In particular, $\theta(M_1)$ is closed in $M_2$ with respect to $\cl_{M_2}$. We write $M_1 \closed M_2$ for a closed embedding.

Given a \qps\ $M$, let $\K^-(M)$ be the smallest class of $L$-structures which
contains $M$ and all its closed substructures and is closed under isomorphism,
and let $\K(M)$ be the smallest class containing $\K^-(M)$ which is also
closed under taking unions of chains of closed embeddings. Then both $\K^-(M)$
and $\K(M)$ satisfy axioms 0, I, and II of quasiminimal excellent classes from
\cite{OQMEC}, and $\K(M)$ also satisfies axiom IV and, together with closed
embeddings, forms an abstract elementary class. We call any class of the form
$\K(M)$  a \emph{quasiminimal class}.

Our main result is:
\begin{theorem}\label{cat theorem}
  If $\K$ is a quasiminimal class then every structure $A \in \K$ is a weakly
  \qps, and up to isomorphism there is exactly one structure in $\K$ of each
  cardinal dimension. In particular, $\K$ is uncountably categorical.
  Furthermore, $\K$ is the class of models of an $\Looq$ sentence.
\end{theorem}

When $M$ satisfies an additional property called \emph{excellence},
Theorem~\ref{cat theorem} is Zilber's main categoricity theorem, specifically
in the form from \cite[Theorem~4.2 and Corollary~5.7]{OQMEC}, along with the
$\Looq$-definability result \cite[Theorem~5.5]{OQMEC}. We will prove
Theorem~\ref{cat theorem} by showing in Proposition~\ref{excellence} that the
specific form of  the excellence property used in the categoricity proof
always holds. 

Assuming that Proposition, we prove the main theorem.
\begin{proof}[Proof of Theorem~\ref{cat theorem}]

  Let $\monster$ be a \qps\ and $\K = \K(\monster)$. As in
  \cite[Theorem~2.2]{OQMEC}, all closed subsets of $\monster$ of dimension
  $\aleph_0$ are isomorphic to each other, and are also \qps s. Let $M$ be one.
  Then by Proposition~\ref{excellence}, $M$ satisfies the excellence
  property. Thus by \cite[Corollary~5.7 and Theorem~4.2]{OQMEC}, $\K(M)$ is a
  quasiminimal excellent class and has exactly one model of each cardinal
  dimension, and by \cite[Theorem~5.5]{OQMEC} it is the class of models of an
  $\Looq$ sentence. It remains to show that $\K = \K(M)$. Let $B$ be a basis for
  $\monster$, and note that $\monster = \bigcup \class{\cl(B')}{B' \subs B, |B'|
  = \aleph_0} $. Since $\K(M)$ is closed under unions of chains it is also
  closed under unions of directed systems, and hence $\monster \in \K(M)$. Thus
  $\K(\monster) = \K(M)$.
\end{proof}

\section{Models and types}

Let $\K$ be a quasiminimal class. We call the structures in $\K$
\emph{models}. Then by \cite[Theorem~2.1]{OQMEC}, the models of dimension up
to $\aleph_1$ are determined up to isomorphism by their dimension.
Furthermore, back-and-forth arguments as in the proof of that theorem
immediately give us the following characterization of types.

\begin{lemma}\label{Galois} 
  Let $\monster$ be a model of dimension $\le \aleph_1$, let $M \closed
  \monster$ with $M$ countable, let $H = \emptyset$ or $H \closed M$, and let
  $\abar, \bbar$ be $n$-tuples from $M$. Then the following are equivalent.
  \begin{itemize} 
    \item $\tp(\abar/H) = \tp(\bbar/H)$.
    \item There exists $f\in\Aut(M/H)$ with $f(\aaa)=\bee$. 
    \item There exists $f\in\Aut(\monster/H)$ with $f(\aaa)=\bee$.
    \item There exists $f\in\Aut(\monster/H)$ with $f(\aaa)=\bee$ and $f(M)=M$.
      \qed
  \end{itemize}
\end{lemma}

Thus Galois types coincide with syntactic types for types over the empty set and
over models, and furthermore Galois types do not depend on the model in which
they are calculated (we have shown this for models of dimension at most
$\aleph_1$, but it will follow from our main result that it holds for arbitrary
models).






\section{Splitting of types}\label{splitting section}
\begin{defn}
  Let $\monster$ be a model, and let $B\subs\monster$ and $\aaa\in\monster$.
  We say that $\tp(\aaa/B)$ \emph{splits} over a finite $A\subs
  B$ if there are finite tuples $\cee$ and $\dee$ in $B$ with 
  \[\tp(\cee/A)=\tp(\dee/A)\qquad\mbox{ but}\]
  \[\tp(\cee/A\cup\aaa)\neq\tp(\dee/A\cup\aaa).\]
\end{defn}
 
\begin{prop}\label{splittaus}
  Let $\monster$ be a model and $M\closed \monster$ be countable closed submodel.
  For each finite tuple $\aaa\in\monster$ there is a finite $A\subs M$ such
  that $\tp(\aaa/M)$ does not split over $A$. 
\end{prop}
\begin{proof}
  Replacing $\monster$ with $\cl(M\aaa)$, we may assume $\monster$ to be countable.
  If $M$ has finite cardinality, we may take $A=M$. So assume $|M|=\aleph_0$.

  We suppose that no such finite $A$ exists and construct uncountably many types
  over $M$, all realised in $\monster$. This contradicts the countability of
  $\monster$. 

  Enumerate $M=\{e_n:n<\omega\}$. For each $k<\omega$  and $\eta: k\to 2$ we
  denote by $\eta^{\frown}0$ and $\eta^{\frown}1$  the functions with domain
  $k+1$ extending $\eta$ and mapping $k$ to 0 and 1 respectively. Given any
  function $f$ and a subset $A$ of the domain of $f$ we write $f|A$ for the
  restricted function.

  We recursively construct finite sets $A_\eta$ and automorphisms
  $\sigma_\eta\in\Aut(\monster)$ such that:
  \begin{enumerate}
    \item $\sigma_\eta(M)=M$.
    \item $\eta\subset\tau$ implies $\sigma_{\tau} | A_{\eta}=\sigma_\eta | A_{\eta}$. 
    \item For any $\mu:\omega\to 2$, we have that  
      \[M=\bigcup_{k<\omega}A_{\mu | k}\]
      and that
      \[M=\bigcup_{k<\omega}\sigma_{\mu|k}(A_{\mu | k}).\]
    \item $\tp(\sigma_{\eta^{\frown}0}(\aaa)/B_\eta)\neq \tp(\sigma_{\eta^{\frown}1}(\aaa)/B_\eta)$ where 
      \[B_\eta=\sigma_{\eta^{\frown}0}(A_{\eta^{\frown}0})\cap \sigma_{\eta^{\frown}1}(A_{\eta^{\frown}1})\subset M.\]
  \end{enumerate}

  First let $A_{\emptyset}=\emptyset$ and $\sigma_{\emptyset}=Id_{\monster}$.
  Then assume we have defined these for all $\eta$ with domain $\leq k$. 

  Since $\tp(\aaa/M)$ splits over $A_{\eta}$ by assumption, there are finite
  $\cee,\dee\in M$ with $$\tp(\cee/A_{\eta})=\tp(\dee/A_{\eta})\textrm{ but}$$
  $$\tp(\cee/A_{\eta}\cup\aaa)\neq\tp(\dee/A_{\eta}\cup\aaa).$$ Hence there is
  $f\in\Aut(\monster/A_{\eta})$ mapping $\cee$ to $\dee$ and by Lemma
  $\ref{Galois}$ we may assume that $f(M)=M$.

  Let $\sigma_{\eta^{\frown}0}=\sigma_\eta$ and $\sigma_{\eta^{\frown}1}=
  \sigma_\eta\circ f$. Furthermore, for $i=0,1$ let
  $$A_{\eta^{\frown}i}=A_{\eta}\cup \{e_{k+1}, \sigma^{-1}_{\eta^{\frown}i}(e_{k+1}), \dee,\cee\}.$$
   We have that
  $$\sigma_{\eta^{\frown}1} | A_{\eta} = \sigma_{\eta^{\frown}0} | A_{\eta} =\sigma_\eta | A_{\eta}$$
  and that
  $$\sigma_{\eta^{\frown}1}(\cee)=\sigma_\eta(\dee)=\sigma_{\eta^{\frown}0}(\dee).$$
  Hence $\sigma_\eta(A_{\eta})$ and $\sigma_\eta(\dee)$ are in the set $B_\eta$
  of item 4.

  Now item 4 must hold, since if there were
  $g\in\Aut(\monster/\sigma_\eta(A_{\eta})\cup \sigma_\eta(\dee))$ mapping
  $\sigma_{\eta^{\frown}0}(\aaa)$ to $\sigma_{\eta^{\frown}1}(\aaa)$, the
  automorphism $\sigma^{-1}_{\eta^{\frown}1}\circ g\circ
  \sigma_{\eta^{\frown}0}$ would map $\dee$ to $\cee$ and fix $\aaa$ and
  $A_\eta$, contradicting splitting.

  Finally we define for each $\mu:\omega\to 2$ a map $f_{\mu}$ as the union of
  the restricted maps $\sigma_{\mu |k}$ on $A_{\mu |k}$ for $k<\omega$. By item
  2 the map is well-defined and by item 3 it is an automorphism of $M$. By
  Lemma~\ref{Galois}, each $f_{\mu}$ extends to an automorphism $\pi_\mu$ of
  $\monster$.

  Now suppose $\mu,\nu : \omega \to 2$ are distinct, let $k$ be greatest such
  that $\mu |k = \nu | k$, and let $\eta = \mu | k$. Then without loss of
  generality, $\mu|k+1 = \eta\concat 0$ and $\nu | k+1 = \eta\concat1$. Thus
  $\pi_\mu | A_{\eta\concat0} = \sigma_{\eta\concat0}|A_{\eta\concat0}$, so
  \[\tp(\pi_\mu(\abar)/\sigma_{\eta\concat0}(A_{\eta\concat0})) =
  \tp(\pi_\mu(\abar)/\pi_\mu(A_{\eta\concat0})) =
  \tp(\sigma_{\eta\concat0}(\abar)/\sigma_{\eta\concat0}(A_{\eta\concat0})) \]
  Since $B_\eta \subs \sigma_{\eta\concat0}(A_{\eta\concat0})$ we have
  $\tp(\pi_\mu(\abar)/B_\eta) = \tp(\sigma_{\eta\concat0}(\abar)/B_\eta)$.

  The same argument shows that
  $\tp(\pi_\nu(\abar)/B_\eta) = \tp(\sigma_{\eta\concat1}(\abar)/B_\eta)$.

  Thus, by item 4, $\tp(\pi_\mu(\abar)/B_\eta) \neq \tp(\pi_\nu(\abar)/B_\eta)$
  and hence $\tp(\pi_\mu(\abar)/M)\neq\tp(\pi_{\nu}(\abar)/M)$. Thus we have
  $2^{\aleph_0}$ different types over $M$, all realised in $\monster$, a
  contradiction.
\end{proof}

\section{Isolation of types}
\begin{defn}
  Let $A$ be a subset of a model $\monster$ and let $\abar\in\cl(A)$. We say that
  the $\tp(\abar/A)$ is \emph{s-isolated} if there is a finite subset $A_0 \subs
  A$ such that if $\bbar \in \cl(A)$ and $\tp(\bbar/A_0) = \tp(\abar/A_0)$
  then $\tp(\bbar/A) = \tp(\abar/A)$. In this case we also say $\tp(\abar/A)$ is
  \emph{s-isolated over $A_0$}. 
\end{defn}

%

In Shelah's notation this is $F^s_{\aleph_0}$-isolation
\cite[p157]{ShelahClassificationTheory}. In general it does not imply
isolation of a type by a single formula, at least not without expanding the
language.

We show that types of tuples inside the closure of a model union a finite set
are s-isolated.

\begin{prop}\label{isolation in closure}
  Let $\monster$ be a model, and let $M\closed \monster$ be a countable closed
  submodel.
  Let $\abar$, $\bbar\in \monster$ be finite tuples with $\bbar \in
  \cl(M\abar)$. Then $\tp(\bbar/M\cup\abar)$ is s-isolated.
\end{prop}

To show that the hypotheses cannot be significantly weakened, consider a \qps\
$\monster$ where the language contains a single equivalence relation, and
$\monster$ has $\aleph_0$ equivalence classes, all of size $\aleph_0$. For $A
\subs \monster$, $\cl(A)$ is the union of the equivalence classes which meet
$A$. Then if $M \subs \monster$ is infinite but not closed, the conclusion
fails.

\begin{proof}[Proof of Proposition~\ref{isolation in closure}]
  By Proposition \ref{splittaus} there exists a finite $A\subset M$ such that
  $\tp(\aaa\bee/M)$ does not split over $A$. We may suppose (extending $A$)
  that $\aaa$ is $\cl$-independent from $M$ over $A$,
  and that $\bee\in\cl(A\aaa)$.
  We will show that $\tp(\bee/M\aaa)$ is s-isolated over $A\aaa$.

  Let $\bee'\in \monster$ with $\tp(\bee'/A\aaa)=\tp(\bee/A\aaa)$.
  Let $\dee\in M$.
  \begin{claim}
    There exists $\dee'\in M$ such that $\tp(\bee'\dee/A\aaa) =
    \tp(\bee\dee'/A\aaa)$.
  \end{claim}
  Assume the claim. Then $\tp(\dee/A)=\tp(\dee'/A)$, so by non-splitting
  $\tp(\dee/A\aaa\bee)=\tp(\dee'/A\aaa\bee)$. Hence
  $\tp(\bee\dee/A\aaa)=\tp(\bee\dee'/A\aaa)=\tp(\bee'\dee/A\aaa)$, so
  $\tp(\bee/A\dee\aaa) = \tp(\bee'/A\dee\aaa)$. So
  $\tp(\bee/M\aaa)=\tp(\bee'/M\aaa)$.

  It remains to prove the claim. Say $\dee=\dee_1\dee_2$ with $\dee_1$ an
  independent tuple over $\cl(A\aaa)$ and $\dee_2\in\cl(A\aaa\dee_1)$. By the
  independence of $\aaa$ from $M$ over $A$, in fact $\dee_2\in\cl(A\dee_1)$.
  Since $\bee\in\cl(A\aaa)$, by (QM4) we have
  $\tp(\bee'/A\aaa\dee_1) = \tp(\bee/A\aaa\dee_1)$.
  So by Lemma~\ref{homogeneity over emptyset}, there exists $\dee'_2\in\monster$
  such that $\tp(\bee'\dee_2/A\aaa\dee_1) = \tp(\bee\dee'_2/A\aaa\dee_1)$.
  But then $\dee'_2\in\cl(A\dee_1)\subs M$, so we conclude by setting
  $\dee' = \dee_1\dee'_2$.
\end{proof}

We remark that the conclusion of Proposition~\ref{isolation in closure}, or that
of Proposition~\ref{splittaus}, could replace $\aleph_0$-homogeneity over models
in the definition of a quasiminimal pregeometry structure:

\begin{cor}\label{wstabversions}
  Let $\monster$ be an $L$-structure for a countable language $L$, equipped with a
  pregeometry $\cl$. Suppose $\monster$ satisfies (QM1)-(QM4) and is
  $\aleph_0$-homogeneous over $\emptyset$, that is, the conclusion of Lemma~\ref{homogeneity over emptyset} holds.

  Then the following are equivalent:
  \begin{enumerate}[(a)]
    \item $\monster$ satisfies (QM5);
    \item If $M\closed \monster$ is countable and $\abar\in \monster$, then there is a
      finite set over which $\tp(\abar/M)$ doesn't split.
    \item If $M\closed \monster$ is countable and $\abar,\bbar\in \monster$ with
      $\bbar\in\cl(M\abar)$, then $\tp(\bbar/M\abar)$ is s-isolated.
  \end{enumerate}
\end{cor}
\begin{proof}
  \providecommand{\dbar}{\bar{d}}
  Proposition~\ref{splittaus} gives (a) $\implies$ (b),
  and the proof of Proposition~\ref{isolation in closure} gives (b) $\implies$ (c).

  We show (c) $\implies$ (a). Let $H,\bbar,H',\bbar',a$ be as in (QM5), with
  $H$ and $H'$ closed in $\monster$ of dimension $\leq\aleph_0$. Write $\sigma : H\bbar
  \to H'\bbar'$ for the given isomorphism. By (c) applied to $H\closed \monster$,
  there exists $\cbar\in H$ such that $\tp(a/H\bbar)$ is
  isolated by $\tp(a/\cbar\bbar)$; let $\cbar':=\sigma(\cbar)$. By
  $\aleph_0$-homogeneity over $\emptyset$, there exists $a'$ such that
  $\tp(\cbar\bbar a) = \tp(\cbar'\bbar' a')$. Now suppose $\dbar\in H$, and
  let $\dbar':=\sigma(\dbar)$. By $\aleph_0$-homogeneity over $\emptyset$
  there exists $a''$ such that $\tp(\dbar'\cbar'\bbar' a') =
  \tp(\dbar\cbar\bbar a'')$, and by the isolation $\tp(\dbar\cbar\bbar a'') =
  \tp(\dbar\cbar\bbar a)$. So $\tp(H\bbar a) = \tp(H'\bbar' a')$ as required.
\end{proof}

\section{Excellence}\label{excellence section}
Shelah's notion of excellence says that types over certain configurations we
call \emph{crowns} are determined over finite sets, i.e. s-isolated. It will be convenient to use
notation for crowns which is borrowed from the notation used in simplicial
complexes, in particular with the use of a boundary operator $\bdy$.

Let $M$ be an infinite-dimensional model, let $B \subs M$ be an independent
subset of cardinality $\aleph_0$, write $M_B = \cl(B)$, and let
$b_1,\ldots,b_n \in B$ be distinct. We define $\bdy_i M_B = \cl (B \minus
\{b_i\})$ and the $n$-crown $\bdy M_B = \bigcup_{i=1}^n \bdy_i M_B$. Note that
$\bdy M_B$ depends on $n$ and the choice of $b_1,\ldots,b_n$, but we suppress
that from the notation.

\begin{defn}
  $M$ is \emph{excellent} if for every  $n \in \N$ with
  $n \geqslant 2$ and every $n$-crown $\bdy M_B$ in $M$, and every finite tuple
  $\abar \in M_B$, the type $\tp(\abar/\bdy M_B)$ is s-isolated.
\end{defn}

Note that the definition of crown here, and consequently the definition of
excellence, is a special case of the definition in \cite{OQMEC}. However, it is
exactly the special case which is used in the proofs in that paper.


\begin{prop}\label{excellence}
  For each $n\geq 2$, each $n$-crown $\bdy M_B$ and $\abar \in M_B$ we have
  \begin{enumerate}[i)]
    \item $\tp(\abar/\bdy M_B)$ is s-isolated, and
    \item If $\tp(\cbar/\bdy M_B) = \tp(\abar/\bdy M_B)$ then there is $\pi \in
      \Aut(M_B/\bdy M_B)$ such that $\pi(\abar) = \cbar$.
  \end{enumerate}
  In particular, $M$ is excellent.
\end{prop}

\begin{proof}
  Any two $n$-crowns in $M$ are isomorphic, so we may fix $B$ and assume $M =
  M_B$. We proceed by induction on $n$. The proofs for the base case $n=2$ and
  the inductive step are very similar, so we do them together. Thus we suppose
  the result holds up to $n-1$ for some $n \geqslant 2$.

  Fix $b_1,\ldots,b_n \in B$, and let $\abar \in M$ be a finite tuple. Choose
  $b_0 \in B \minus\{b_1,\ldots,b_n\}$ such that $\abar \in \cl(B \minus
  \{b_0\})$ and let $M' = \cl(B \minus \{b_0\})$. Choose $\pi \in \Aut(M/\cl(B
  \minus\{b_0,b_n\}))$ such that $\pi(b_n) = b_0$ and $\pi(b_0)=b_n$.

  First, suppose $n=2$. Then $\abar, \pi(\abar) \in \cl(\bdy_1 M, b_1)$ so, by
  Proposition~\ref{isolation in closure}, $\tp(\abar, \pi(\abar) /\bdy_1 M,
  b_1)$ is s-isolated.

  Now suppose $n > 2$. Then $\Lambda := \bigcup_{i=1}^{n-1} \bdy_i M$ is an $(n-1)$-crown
  and $\cl(\Lambda) = M$ so, by part i) of the
  induction hypothesis, $\tp(\abar, \pi(\abar) / \Lambda)$
  is s-isolated. Note that in this case $b_1 \in \Lambda$,
  so $\tp(\abar, \pi(\abar) / b_1,\Lambda)$ is
  s-isolated. Thus (whatever $n$ is)  $\tp(\abar / b_1, \pi(\abar),
  \bigcup_{i=1}^{n-1} \bdy_i M)$ is s-isolated, say over $A_0$. Since
  $\pi(\abar)\subs \bdy_n M$, we have $A_0\subs \bdy M$.

  We proceed to show that $\tp(\abar/\bdy M)$ is s-isolated over
  $A_0$. So suppose $\cbar\in M$ and $\tp(\cbar/A_0) = \tp(\abar/A_0)$. Then we have $\tp(\cbar / b_1, \pi(\abar),
  \bigcup_{i=1}^{n-1} \bdy_i M) = \tp(\abar / b_1, \pi(\abar),
  \bigcup_{i=1}^{n-1} \bdy_i M)$, so by Lemma~\ref{Galois} (if $n=2$) or by part
  ii) of the inductive hypothesis (if $n>2$) there is $\sigma \in \Aut(M/b_1,
  \pi(\abar), \bigcup_{i=1}^{n-1} \bdy_i M)$ such that $\sigma(\abar) = \cbar$.
  Let $\eta$ be the commutator $\eta = \sigma \pi^{-1}\sigma^{-1}\pi$. Then
  since $\pi(\abar)$ is fixed by $\sigma^{-1}$, we have $\eta(\abar) =
  \sigma(\abar) = \cbar$. Now $\eta$ fixes $\bigcup_{i=1}^n \cl(B \minus
  \{b_0,b_i\}) = \bdy M'$ pointwise; indeed, for $i=1,\ldots,n-1$ we have $\cl(B
  \minus \{b_0,b_i\}) \subseteq \cl(B \minus \{b_i\})$, and the latter is fixed
  setwise by $\pi$ and pointwise by $\sigma$, while $\cl(B \minus \{b_0,b_n\})$
  is fixed pointwise by $\pi$ and setwise by $\sigma$. So $\tp(\cbar/\bdy M') =
  \tp(\abar/ \bdy M')$.

  Let $B_0$ be a finite subset of $B \minus\{b_0,b_1,\ldots,b_n\}$ such that
  $\abar \in \cl(B_0 \cup \{b_1,\ldots,b_n\})$. Then also $\cbar \in
  \cl(B_0\cup \{b_1,\ldots,b_n\})$, since $B_0\cup\{b_1,\ldots,b_n\} \subs
  \bigcup_{i=1}^{n-1}\bdy_i M \cup \{b_1\}$.

  Since $B$ is infinite, there is a bijection $B \minus (B_0 \cup
  \{b_0,b_1,\ldots,b_n\}) \to B\minus(B_0 \cup \{b_1,\ldots,b_n\})$ extending
  to an isomorphism $\phi: M' \to M$ which fixes $\cl(B_0 \cup
  \{b_1,\ldots,b_n\})$ pointwise.

  Now $\tp(\cbar/\bdy M') = \tp(\abar/\bdy M')$, so $\tp(\phi(\cbar)/\phi(\bdy
  M')) = \tp(\phi(\abar)/\phi(\bdy M'))$; but $\phi(\abar) = \abar$ and
  $\phi(\cbar) = \cbar$, and $\phi(\bdy M') = \bdy M$, so $\tp(\cbar/\bdy M) =
  \tp(\abar/\bdy M)$. Thus part i) is proved.

  For ii), suppose that $\tp(\cbar / \bdy M) = \tp(\abar/\bdy M)$, and so in
  particular $\tp(\cbar/A_0) = \tp(\abar/A_0)$. Let $\eta \in \Aut(M/\bdy M')$
  and $\phi: M'\to M$ be as above. Since $\eta$ fixes $B$ pointwise, $\eta|M'
  \in \Aut(M'/\bdy M')$. Defining $\theta = \phi \circ (\eta|M') \circ
  \phi^{-1}$ we have $\theta \in \Aut(M/\bdy M)$ with $\theta(\abar) = \cbar$,
  which proves ii).
\end{proof}

\section{Quasiminimal structures}\label{sec:qm}
An uncountable structure $M$ is {\em quasiminimal} if every first-order
$M$-definable subset of $M$ is countable or co-countable.
In this section, we treat the question of when a quasiminimal structure
is a quasiminimal pregeometry structure. Certainly some conditions are
required - for example, $\omega_1 \cross \Q$ equipped with the lexicographic
order has quantifier elimination and is quasiminimal, but the quasiminimal
closure ($cl_p$ defined below) does not satisfy exchange.

Based on the analyses of Zilber \cite{ZElemGST} and Pillay-Tanovi\'{c}
\cite{PiTaQM}, we are able to give simple ``natural'' criteria which, under the
assumption of quasiminimality, substitute for all the conditions of quasiminimal
pregeometry structures other than (QM5). For (QM5), we have no alternative
formulation in this context beyond those given in Corollary~\ref{wstabversions}.

So let $M$ be an uncountable quasiminimal structure in a countable language.
Suppose, extending the language if necessary, quantifier elimination for types
realised in $M$: if $\abar\in M$ and $\bbar\in M$ have the same
quantifier-free type, then $\abar$ and $\bbar$ have the same first-order type.

Let $p\in S_1(M)$ be the generic type, the type consisting precisely of the
co-countable formulas. For $A\subseteq M$, define $\cl_p(A) := \{ x\in M \;|\;
x \not\models p_A \}$. A weak Morley sequence in $p$ over $A\subseteq M$ is a
sequence $(a_1,\ldots)$ such that $a_i\in M$ and $\tp(a_i/Aa_{<i}) =
p_{Aa_{<i}}$, where $a_{<i} := \{ a_j \;|\; j<i \}$.

\begin{prop}
  \begin{enumerate}[(A)]
    \item
      $(M,\cl_p)$ is a quasiminimal pregeometry structure if
      \begin{enumerate}[(i)]
	\item $p$ does not split over $\emptyset$; i.e.\ if
	  $\tp(\bbar/\emptyset)=\tp(\bbar'/\emptyset)$ then for all $\phi$ we
	  have $\phi(x,\bbar)\in p$ iff $\phi(x,\bbar')\in p$, i.e.\
	  $|\phi(M,\bbar)| = |\phi(M,\bbar')|$;
	\item There is no $M$-definable partial order on $M$, defined over
	  finite $A\subseteq M$ say, for which weak Morley sequences in $p$
	  over $A$ are increasing.
	\item $(M,\cl_p)$ satisfies (QM5).
      \end{enumerate}

    \item
      Conversely, if $(M,\cl)$ is an uncountable quasiminimal pregeometry
      structure, then $M$ is a quasiminimal structure, $\cl=\cl_p$, and
      (i)-(iii) hold.
  \end{enumerate}
\end{prop}
\begin{proof}
  \begin{enumerate}[(A)]
    \item
      By \cite{PiTaQM}[Theorem 4], (i) and (ii) imply that $\cl_p$ is indeed a
      pregeometry. Axioms (QM1-3) are clear. (QM4) follows directly from (i).
    \item
      By (QM1) and (QM4), there exists a complete type $q\in S_1(M)$ such
      that for $B\subseteq M$ finite and $a\in M$, we have $a\models q_B$ iff
      $a\notin\cl(B)$. By (QM3) and uncountability of $M$, a formula
      $\phi(x,\bbar) \notin q$ iff $\phi(M,\bbar)$ is countable. So $M$ is
      quasiminimal, $q=p$, and hence $\cl=\cl_p$.

      To prove (i), suppose $\phi(M,\bbar)\in p$ and
      $\tp(\bbar/\emptyset)=\tp(\bbar'/\emptyset)$. Then say
      $a\in\phi(M,\bbar)\setminus\cl(\bbar)$. By (QM5), there exists $a'\in M$
      such that $\tp(a,\bbar)=\tp(a',\bbar')$. By (QM1),
      $a'\notin\cl(\bbar')$. Hence $\phi(M,\bbar')\in p$.

      To prove (ii), note that permutations of weak Morley sequences are weak
      Morley sequences, since $\cl_p=\cl$ is a pregeometry.
  \end{enumerate}
\end{proof}

\begin{remark}
  Conditions (i) and (ii) in the preceeding Proposition could be replaced with
  the following conditions of a more elementary flavour, which appear in
  \cite{ZElemGST}:
  \begin{enumerate}[(i')]
    \item $\aleph_0$-homogeneity over $\emptyset$;
    \item
      ``Countability is weakly definable in $M$'': if $a\in\cl_p(\bbar)$, then
      there exists a formula $\phi(x,\ybar)$ over $\emptyset$ such that
      $M\models\phi(a,\bbar)$ and $|\phi(M,\bbar')| \leq \aleph_0$ for all
      $\bbar'\in M$;
    \item
      $|M|>\aleph_1$ or there is no definable partial order on $M$ with a
      chain in $M$ of order type $\omega_1$.
  \end{enumerate}

  Indeed, (i') and (ii') imply (i), since if then $\psi(M,\bbar)$ is countable
  then it is covered by countably many $\phi_i(M,\bbar)$ as in (ii'); by (i'),
  if $\tp(\bbar')=\tp(\bbar)$ then $\psi(M,\bbar')$ is countable since it is
  covered by the countably many countable $\phi_i(M,\bbar')$. By
  \cite{ZElemGST}[Lemma~3.0.3], (ii') and (iii') imply that $\cl_p$ is a
  pregeometry and hence that weak Morley sequences in $p$ are indiscernible,
  and hence that (ii) holds by \cite{PiTaQM}[Theorem 4];
  conversely, \cite{PiTaQM}[Theorem 4] shows under the assumption of (i) and
  (ii) that $p$ is $\emptyset$-definable, which implies (ii'), and that $\cl_p$
  satisfies exchange, which implies (iii').
\end{remark}

\bibliographystyle{alpha}
\bibliography{QMAndE}

\end{document}